\documentclass{article}
\usepackage{amssymb,amsbsy}
\usepackage{amsmath}
\usepackage{amsthm}
\usepackage{mathrsfs}
\usepackage{color}
\usepackage[dvipsnames]{xcolor}
\usepackage{tikz}

\newtheorem{theorem}{Theorem}[section]
\newtheorem{lemma}[theorem]{Lemma}
\newtheorem{proposition}[theorem]{Proposition}
\newtheorem{corollary}[theorem]{Corollary}

\theoremstyle{definition}

\newtheorem{example}[theorem]{Example}

\theoremstyle{remark}
\newtheorem{remark}[theorem]{Remark}

\begin{document}
	
	\title{The extended Frobenius problem for Fibonacci sequences incremented by a Fibonacci number}
	
	\author{Aureliano M. Robles-P\'erez\thanks{Departamento de Matem\'atica Aplicada \& Instituto de Matem\'aticas (IMAG), Universidad de Granada, 18071-Granada, Spain. \newline E-mail: \textbf{arobles@ugr.es} (\textit{corresponding author}); ORCID: \textbf{0000-0003-2596-1249}.}
		\mbox{ and} Jos\'e Carlos Rosales\thanks{Departamento de \'Algebra \& Instituto de Matem\'aticas (IMAG), Universidad de Granada, 18071-Granada, Spain. \newline E-mail: \textbf{jrosales@ugr.es}; ORCID: \textbf{0000-0003-3353-4335}.} }
	
	\date{\today}
	
	\maketitle
	
	\begin{abstract}
		We study the extended Frobenius problem for sequences of the form $\{f_a+f_n\}_{n\in\mathbb{N}}$, where $\{f_n\}_{n\in\mathbb{N}}$ is the Fibonacci sequence and $f_a$ is a Fibonacci number. As a consequence, we show that the family of numerical semigroups associated to these sequences satisfies the Wilf's conjecture.
	\end{abstract}
	
	\medskip
	
	\noindent{\it 2010 AMS Classification:} 11D07, 11B39 (Primary); 11A67, 05A17 (Secondary). 	
	
	\section{Introduction}
	
	Let $S\subseteq \mathbb{N}$ be the set generated by the sequence of positive integers $(a_1,\ldots,a_e)$, that is, $S=\langle a_1,\ldots,a_e \rangle = a_1{\mathbb N}+\cdots+a_e{\mathbb N}$. If $\gcd(a_1,\ldots,a_e)=1$, then it is well known that $S$ has a finite complement in $\mathbb{N}$. This fact leads to the classical problem in additive number theory called the Frobenius problem: what is the greatest integer $\mathrm{F}(S)$ which is not an element of $S$? Although this problem is solved for $e=2$ (see \cite{sylvester}), we have that it is not possible to find a polynomial formula to compute $\mathrm{F}(S)$ if $e\geq3$ (see \cite{curtis}). Therefore, many efforts have been made to obtain partial results or to develop algorithms to get the answer to this question (see \cite{alfonsin}).
	
	Another interesting question is to compute the cardinality $\mathrm{g}(S)$ of the set $\mathbb{N}\setminus S$. In fact, sometimes finding formulas for $\mathrm{F}(S)$ and $\mathrm{g}(S)$ is known as the extended Frobenius problem.
	
	Let us recall that the Fibonacci sequence is given by the recurrence relation $f_{n+2} = f_{n+1} + f_n$ for $n\geq0$ and the initial conditions $f_0=0, f_1=1$. This sequence has been widely studied and is present in many real phenomena (for a popular paper, see \cite{gardner}).
	
	Among others, the main goal of this work is to solve the extended Frobenius problem for $S$ generated by Fibonacci sequences incremented by a Fibonacci number. That is, if $\{f_0,f_1,\ldots,f_n,\ldots\}$ is a Fibonacci sequence and $f_a$ is a Fibonacci number, then we will consider $S(a)=\langle f_a+f_0,f_a+f_1,\ldots,f_a+f_n,\ldots\rangle$. Thus, our work can be considered along the lines of \cite{fel,marin,matthews}. By the way, observe that these authors always consider sequences of three numbers and we do not.	
	
	In order to achieve our purpose, we will use the theory of numerical semigroups (see Section~\ref{ns} for several results of this theory), which is closely related with the Frobenius problem. Indeed, the sets $S(a)$ defined above are numerical semigroups. In Section~\ref{msgS(a)} we will determine the minimal finite subsequence of $\{f_a+f_0,f_a+f_1,\ldots,f_a+f_n,\ldots\}$ which generates $S(a)$; in Section~\ref{aperyS(a)} we will explicitly give the Ap\'ery sets related to our numerical semigroups; and in Sections~\ref{frobeniusS(a)} and \ref{genusS(a)} we will give the formulas for solving the extended Frobenius problem. Finally, as a result derived from our study, in Section~\ref{genusS(a)} we will check that our family of numerical semigroups satisfies the Wilf's conjecture (see \cite{wilf}).

	\section{Preliminaries (on numerical semigroups)}\label{ns}
	
	Let $\mathbb{Z}$ be the set of integers and $\mathbb{N} = \{ z\in\mathbb{Z} \mid z\geq 0\}$. A submonoid of $(\mathbb{N},+)$ is a subset $M$ of $\mathbb{N}$ such that is closed under addition and contains de zero element. A \textit{numerical semigroup} is a submonoid of $(\mathbb{N},+)$ such that $\mathbb{N}\setminus S=\{n\in\mathbb{N} \mid n\not\in S\}$ is finite.
	
	Let $S$ be a numerical semigroup. From the finiteness of $\mathbb{N}\setminus S$, we can define two invariants of $S$. Namely, the \textit{Frobenius number of $S$} is the greatest integer that does not belong to $S$, denoted by $\mathrm{F}(S)$, and the \textit{genus of $S$} is the cardinality of $\mathbb{N}\setminus S$, denoted by $\mathrm{g}(S)$.
	
	If $X$ is a non-empty subset of $\mathbb{N}$, then we denote by $\langle X \rangle$ the submonoid of $(\mathbb{N},+)$ generated by $X$, that is,
	\[ \langle X \rangle=\big\{\lambda_1x_1+\cdots+\lambda_nx_n \mid n\in\mathbb{N}\setminus \{0\}, \ x_1,\ldots,x_n\in X, \ \lambda_1,\ldots,\lambda_n\in \mathbb{N}\big\}. \]
	It is well known (see Lemma~2.1 of \cite{springer}) that $\langle X \rangle$ is a numerical semigroup if and only if $\gcd(X)=1$.
	
	If $S$ is a numerical semigroup and $S=\langle X \rangle$, then we say that $X$ is a \textit{system of generators of $S$}. Moreover, if $S\not=\langle Y \rangle$ for any subset $Y\subsetneq X$, then we say that $X$ is a \textit{minimal system of generators of $S$}. In Theorem~2.7 of \cite{springer} it is shown that each numerical semigroup admits a unique minimal system of generators and that such a system is finite. We denote by $\mathrm{msg}(S)$ the minimal system of generators of $S$. The cardinality of $\mathrm{msg}(S)$, denoted by $\mathrm{e}(S)$, is the \textit{embedding dimension of $S$}.
	
	The (extended) Frobenius problem for a numerical semigroup $S$ consists of finding formulas that allow us to compute $\mathrm{F}(S)$ and $\mathrm{g}(S)$ in terms of $\mathrm{msg}(S)$. As in the case of the Frobenius problem for sequences, such formulas are well known for $\mathrm{e}(S)=2$ (see \cite{sylvester}), but it is not possible to find polynomial formulas when $e(S)\geq3$ (see \cite{curtis}), except for particular families of numerical semigroups.
	
	If $n\in S\setminus\{0\}$, then a very useful tool to describe a numerical semigroup $S$ is the set $\mathrm{Ap}(S,n)=\{s\in S \mid s-n\not\in S\}$, called the \textit{Ap\'ery set of $n$ in $S$} (after \cite{apery}). The following result is Lemma~2.4 of \cite{springer}.
	
	\begin{proposition}\label{prop02}
		Let $S$ be a numerical semigroup and $n\in S\setminus\{0\}$. Then the cardinality of $\mathrm{Ap}(S,n)$ is $n$. Moreover,
		\[ \mathrm{Ap}(S,n)=\{w(0)=0, w(1), \ldots, w(n-1)\}, \] 
		where $w(i)$ is the least element of $S$ congruent with $i$ modulo $n$.
	\end{proposition}
	
	The knowledge of $\mathrm{Ap}(S,n)$ allows us to solve the problem of membership of an integer to the numerical semigroup $S$. In fact, if $x\in\mathbb{Z}$, then $x\in S$ if and only if $x\geq w(x\bmod n)$. Moreover, we have the following result from \cite{selmer}.
	
	\begin{proposition}\label{prop03}
		Let $S$ be a numerical semigroup and let $n\in S\setminus\{0\}$. Then
		\begin{enumerate}
			\item $\mathrm{F}(S)=\max(\mathrm{Ap}(S,n))-n$,
			\item $\mathrm{g}(S)=\frac{1}{n}(\sum_{w\in \mathrm{Ap}(S,n)} w)-\frac{n-1}{2}$.
		\end{enumerate}
	\end{proposition}
	
	From this proposition, it is clear that, if we have an explicit description of $\mathrm{Ap}(S,n)$, then we have the solution of the Frobenius problem for $S$.

	\section{The minimal system of generators of $S(a)$}\label{msgS(a)}
	
	From the definition of $S(a)$, it is clear that, if $a\in\{0,1,2\}$, then $S(a)=\mathbb{N}$. Therefore, in what follows, and unless otherwise indicated, we will assume that $a\in\mathbb{N}\setminus\{0,1,2\}$.
	
	In this section our main objective will be to determine the minimal system of generators of $S(a)=\langle f_a+f_0, f_a+f_1, \ldots, f_a+f_n \ldots\rangle$.
	
	First of all, let us observe that $\gcd\{f_a+f_0, f_a+f_1\} = \gcd\{f_a, f_a+1\} = 1$ and, therefore, $S(a)$ is a numerical semigroup.
	
	Let us see several results that are necessary to achieve our purpose.
	
	\begin{lemma}{\cite[p. 107]{honsberger}}\label{lem01}
		If $i\in\mathbb{N}$, then $f_{a+i} = f_{i+1}f_a + f_i f_{a-1}$.
	\end{lemma}
	
	
	\begin{lemma}\label{lem02}
		If $i\in\mathbb{N}$, then $f_a + f_{a+i} \in\langle f_a+f_0, f_a + f_{a-1} \rangle$.
	\end{lemma}
	
	\begin{proof}
		By Lemma~\ref{lem01}, we have that $f_a + f_{a+i} = (f_{i+1}+1)f_a + f_i f_{a-1}$. Now, since $f_0=0$, then $f_a + f_{a+i} = (f_{i-1}+1)(f_a+f_0) + f_i (f_a + f_{a-1})$. Consequently, $f_a + f_{a+i} \in\langle f_a+f_0, f_a + f_{a-1} \rangle$.
	\end{proof}
	
	The following result is Lemma~2.3 of \cite{springer}.
	
	\begin{lemma}\label{lem03}
		If $S$ is a numerical semigroup and $S^*=S\setminus\{0\}$, then $\mathrm{msg}(S) = S^* \setminus (S^* + S^*)$.
	\end{lemma}
	
	If $S$ is a numerical semigroup, then the \textit{multiplicity of $S$} is the least positive integer belonging ot $S$, denoted by $\mathrm{m}(S)$.
	
	The next lemma is an immediate consequence of Lemma~\ref{lem03}.
	
	\begin{lemma}\label{lem04}
		If $X$ is a system of generators of a numerical semigroup $S$ and $X\subseteq \{\mathrm{m}(S), \mathrm{m}(S)+1, \ldots, 2\mathrm{m}(S)-1 \}$, then $X=\mathrm{msg}(S)$.
	\end{lemma}
	
	We are now ready to show the announced result on the minimal system of generators of $S(a)$.
	
	\begin{proposition}\label{prop05}
		We have that $\mathrm{msg}(S(a)) = \{ f_a+f_0, f_a+f_2, \ldots, f_a+f_{a-1}\}$.
	\end{proposition}
	
	\begin{proof}
		By Lemma~\ref{lem02}, we deduce that $\{ f_a+f_0, f_a+f_2, \ldots, f_a+f_{a-1}\}$ is a system of generators of $S(a)$. Since $\mathrm{m}(S(a)) = f_a+f_0 = f_a$ and $f_a = f_a+f_0 < f_a+f_2 < \ldots  < f_a+f_{a-1} < 2f_a$, by applying Lemma~\ref{lem04}, we conclude the proof. 
	\end{proof}
	
	As an immediate consequence of the previous proposition, we have the following result.
	
	\begin{corollary}\label{cor06}
		The embedding dimension of $S(a)$ is $\mathrm{e}(S(a))=a-1$.
	\end{corollary}
	
	\begin{example}\label{exmp07}
		By definition, $S(7) = \langle 13+0, 13+1, 13+2, 13+3, 13+5, 13+8, 13+13, 13+21, 13+34, \ldots \rangle$. By Proposition~\ref{prop05}, we know that $\mathrm{msg}(S(7)) = \{13,14,15,16,18,21\}$ and, therefore, $\mathrm{e}(S(7))=6$.
	\end{example}
	
	It is clear that $\{ f_n \mid n\geq a \} \subseteq \langle f_a, f_{a+1} \rangle$ and that $\{ f_a, f_{a+1} \} \subseteq S(a)$. Therefore, we have the next result.
	
	\begin{proposition}\label{prop08}
		We have that $\{ f_n \mid n\geq a \} \subseteq S(a)$.
	\end{proposition}

	\section{The Ap\'ery set of $S(a)$}\label{aperyS(a)}
	
	Our main objective in this section is to prove Theorem~\ref{thm12}, which describes $\mathrm{Ap}(S(a), f_a)$.
	
	It is a well known fact that every non-negative integer can be uniquely represented as a sum of non-consecutive Fibonacci numbers (see \cite{zeckendorf}), the so-called \textit{Zeckendorf decomposition}. Moreover, this decomposition is minimal in the sense that no other decomposition has fewer summands (see \cite{decomposition}). Both facts are summarised in the following lemma.
	
	\begin{lemma}\label{lem00}
		If $x\in\mathbb{N}\setminus\{0\}$, then there exists a unique $k\in\mathbb{N}\setminus\{0,1\}$ such that $x=\sum_{i=2}^k b_i f_i$ with $(b_2,\ldots,b_k)\in\{0,1\}^{k-1}$, $b_k=1$, and $b_ib_{i+1}=0$ for all $i\in\{2,\ldots,k-1\}$. Moreover, if $x=\sum_{i=2}^{k'} c_i f_i$ with $(c_2,\ldots,c_{k'})\in\mathbb{N}^{k'-1}$ and $k'\in\mathbb{N}$, then $\sum_{i=2}^k b_i \leq \sum_{i=2}^{k'} c_i$.
	\end{lemma}

	If $x\in\mathbb{N}$, then we denote by 
	\[ \beta(x) = \min\left\{ \sum_{i=2}^{l} b_i \;\bigg\vert\; x=\sum_{i=2}^{l} b_i f_i, \mbox{ with } (b_2,\ldots,b_l)\in\mathbb{N}^{l-1}, \; l\geq2 \right\}. \]
	
	\begin{remark}\label{rem00}
		By Lemma~\ref{lem00}, it is clear that, if $x=\sum_{i=2}^k b_i f_i$ is the Zeckendorf decomposition of $x\in\mathbb{N}\setminus\{0\}$, then $\beta(x)=\sum_{i=2}^k b_i$. Moreover, $\beta(0)=0$.
	\end{remark}
	
	To prove Theorem~\ref{thm12} we need the following result.
	
	\begin{lemma}\label{lem09}
		 If $a\in\mathbb{N}\setminus\{0,1,2\}$, $(d_2,\ldots,d_{a-1})\in\mathbb{N}^{a-2}$, and $\sum_{i=2}^{a-1} d_i f_i \geq f_a$, then there exists $(c_2,\ldots,c_{a-1})\in\mathbb{N}^{a-2}$ such that $\sum_{i=2}^{a-1} d_i f_i = f_a + \sum_{i=2}^{a-1} c_i f_i$ and $\sum_{i=2}^{a-1} c_i < \sum_{i=2}^{a-1} d_i$.
	\end{lemma}
	
	We will show the proof of the above lemma in two steps. In the first (Lemma~\ref{lem10}) we obtain the result directly for some cases. In the second (Lemma~\ref{lem11}) we prove it by mathematical induction for the rest of the cases.
	
	\begin{lemma}\label{lem10}
		Let $a\in\mathbb{N}\setminus\{0,1,2,3\}$ and $\sum_{i=2}^{a-1} b_i f_i \geq f_a$, with $(b_2,\ldots,b_{a-1})\in\mathbb{N}^{a-2}$. If ($b_{a-2}\geq 1$ and $b_{a-1}\geq 1$) or ($b_{a-2}=0$ and $b_{a-1}\geq 2$), then we have that $\sum_{i=2}^{a-1} b_i f_i - f_a = \sum_{i=2}^{a-1} c_i f_i$, with $(c_2,\ldots,c_{a-1})\in\mathbb{N}^{a-2}$ and $\sum_{i=2}^{a-1} c_i < \sum_{i=2}^{a-1} b_i$.
	\end{lemma}
	
	\begin{proof}
		Let us observe that 
		\[ \sum_{i=2}^{a-1} b_i f_i - f_a = \sum_{i=2}^{a-1} b_i f_i - f_{a-2} - f_{a-1} = \sum_{i=2}^{a-1} b_i f_i + f_{a-3} - 2f_{a-1}.\]
		
		Now, if $b_{a-2}\geq 1$ and $b_{a-1}\geq 1$, then $\sum_{i=2}^{a-1} b_i f_i - f_a = \sum_{i=2}^{a-1} c_i f_i$, with $c_i = b_i$ for $2\leq i\leq a-3$, $c_{a-2} = b_{a-2}-1$, and $c_{a-1} = b_{a-1}-1$. Thus, the result is proven in this case.
		
		Similarly, if $b_{a-2}=0$ and $b_{a-1}\geq 2$, then $\sum_{i=2}^{a-1} b_i f_i - f_a = \sum_{i=2}^{a-1} c_i f_i$, with $c_i = b_i$ for $2\leq i\leq a-4$, $c_{a-3} = b_{a-3}+1$, $c_{a-2} = b_{a-2} = 0$, and $c_{a-1} = b_{a-1}-2$. So, this case is also proven.
	\end{proof}
	
	\begin{lemma}\label{lem11}
		Let $a\in\mathbb{N}\setminus\{0,1,2\}$ and $\sum_{i=2}^{a-1} b_i f_i \geq f_a$, with $(b_2,\ldots,b_{a-1})\in\mathbb{N}^{a-2}$. Then we have that $\sum_{i=2}^{a-1} b_i f_i - f_a = \sum_{i=2}^{a-1} c_i f_i$, with $(c_2,\ldots,c_{a-1})\in\mathbb{N}^{a-2}$ and $\sum_{i=2}^{a-1} c_i < \sum_{i=2}^{a-1} b_i$.
	\end{lemma}
	
	\begin{proof}
		We are going to prove the lemma using induction on $a$. 
		
		\textit{(Basis.)} We first analyse the cases $a=3$ and $a=4$.
		
		Let us take $a=3$. Then $\sum_{i=2}^{a-1} b_i f_i =b_2f_2=b_2$. Thus, having in mind that $f_2=1$, if $b_2 f_2 \geq f_3$, then $b_2 \geq f_3$. Therefore, $b_2f_2-f_3 = c_2f_2$ with $c_2=b_2-f_3$.
		
		Now, if $a=4$, then $\sum_{i=2}^{a-1} b_i f_i =b_2f_2 + b_3f_3$. Since $f_2=1$, $f_3=2$ and $f_4=3$, we have that, if $b_2f_2 + b_3f_3 \geq f_4$, then $b_2 + 2b_3 \geq 3$. Consequently, $(b_2,b_3)\in B=\mathbb{N}^2\setminus \{(0,0), (1,0), (0,1), (2,0) \}$. Let us see two particular cases of elements in $B$.
		\begin{enumerate}
			\item If $(b_2,b_3)=(k,0)$, $k\geq3$, then $b_2 f_2 + b_3 f_3 - f_ 4 = c_2 f_2 + c_3 f_3$ with $c_2=b_2-3$ and $c_3=0$.
			\item In any other case we apply Lemma~\ref{lem10}.
		\end{enumerate}
		
		\textit{(Induction hypothesis.)} We now suppose that $a\geq 5$, $\sum_{i=2}^{a-1} b_i f_i \geq f_a$, and that the statement is true for all $k\in\{3,4,\ldots,a-1\}$.
		
		\textit{(Induction step.)} In the light of Lemma~\ref{lem10}, we need only consider three cases. Moreover, we recall that $\sum_{i=2}^{a-1} b_i f_i - f_a = \sum_{i=2}^{a-1} b_i f_i - f_{a-2} - f_{a-1}$.
		\begin{enumerate}
			\item If $b_{a-2}\geq 1$ and $b_{a-1}=0$, then $\sum_{i=2}^{a-1} b_i f_i - f_a = \sum_{i=2}^{a-2} b'_i f_i - f_{a-1}$, with $b'_i = b_i$ for $2\leq i\leq a-3$ and $b'_{a-2} = b_{a-2}-1$. Now, by the induction hypothesis for $k=a-1$, the result is proven in this case.
			
			\item If $b_{a-2}=0$ and $b_{a-1}=1$, then $\sum_{i=2}^{a-1} b_i f_i - f_a = \sum_{i=2}^{a-3} b_i f_i - f_{a-2}$. Then, by the induction hypothesis for $k=a-2$, the case is proven.
			
			\item If $b_{a-2}=b_{a-1}=0$, then $\sum_{i=2}^{a-1} b_i f_i - f_a = \sum_{i=2}^{a-3} b_i f_i - f_{a-2} - f_{a-1}$. Now, by the induction hypothesis for $k=a-2$ (having in mind that $\sum_{i=2}^{a-3} b_i f_i - f_{a-2} \geq f_{a-1}> 0$) and $k=a-1$, it follows that $\sum_{i=2}^{a-3} b_i f_i - f_{a-2} - f_{a-1} = \sum_{i=2}^{a-2} b'_i f_i - f_{a-1} = \sum_{i=2}^{a-1} c_i f_i$, with $b'_{a-2}=c_{a-1}=0$ and $\sum_{i=2}^{a-1} c_i = \sum_{i=2}^{a-2} c_i < \sum_{i=2}^{a-2} b'_i = \sum_{i=2}^{a-3} b'_i < \sum_{i=2}^{a-3} b_i= \sum_{i=2}^{a-1} b_i$. Therefore, the case is proven.
		\end{enumerate}
	\end{proof}
	
	\begin{theorem}\label{thm12}
		Let $a\in\mathbb{N}\setminus\{0,1,2\}$. If $x\in\{0,1,\ldots,f_a-1\}$ and $\mathrm{Ap}(S(a),f_a) = \{w(0)=0, w(1),\ldots,w(f_a-1)\}$, then $w(x) = \beta(x)f_a+x$.
	\end{theorem}
	
	\begin{proof}
		The result is trivial for $x=0$. So let us suppose that $x\in\{1,\ldots,f_a-1\}$.
		
		If $x=\sum_{i=2}^k b_i f_i$ is the Zeckendorf decomposition of $x$, then $k<a$ and $\beta(x)f_a+x = \sum_{i=2}^{k} b_i (f_a+f_i) \in S(a)$. Moreover, $\beta(x)f_a+x \equiv x\pmod{f_a}$. Therefore, $w(x)\leq \beta(x)f_a+x$.
		
		We now suppose that $w(x)=\sum_{i=2}^{a-1} b'_i (f_a+f_i)$, with $(b'_2,\ldots,b'_{a-1})\in\mathbb{N}^{a-2}$. Obviously, $\sum_{i=2}^{a-1} b'_i f_i = x + \alpha f_a$ with $\alpha \in \mathbb{N}$. If $\alpha \geq 1$, then we can apply Lemma~\ref{lem09} and get that there exists $(c_2,\ldots,c_{a-1})\in\mathbb{N}^{a-2}$ such that $w(x)=\sum_{i=2}^{a-1} c_i (f_a+f_i) + f_a \left(1+\sum_{i=2}^{a-1} (b'_i-c_i)\right)$, with $\sum_{i=2}^{a-1} (b'_i-c_i)>0$. Therefore, $w(x)-f_a \in S(a)$, in contradiction with the fact that $w(x) \in\mathrm{Ap}(S(a),f_a)$. Thus, we have $\sum_{i=2}^{a-1} b'_i f_i=x$ and, in consequence, $w(x)=\left( \sum_{i=2}^{a-1} b'_i \right)f_a+x$. Finally, from the definition of $\beta(x)$, we can easily conclude that $w(x)\geq \beta(x)f_a+x$.
	\end{proof}
	
	\begin{example}\label{exmp-ap}
		By Example~\ref{exmp07}, we have that $S(7)=\langle 13,14,15,16,18,21 \rangle$. Furthermore, from Theorem~\ref{thm12} and the corresponding Zeckendorf decompositions, we deduce that
		\begin{itemize}
			\item $1=f_2;\, 2=f_3;\, 3=f_4;\, 5=f_5;\, 8=f_6 \Rightarrow \beta(1)=\beta(2)=\beta(3)=\beta(5)=\beta(8)=1 \Rightarrow w(1)=14;\, w(2)=15;\, w(3)=16;\, w(5)=18;\, w(8)=21$;
			\item $4=f_4+f_2;\, 6=f_5+f_2;\, 7=f_5+f_3; \, 9=f_6+f_2;\, 10=f_6+f_3;\, 11=f_6+f_4 \Rightarrow \beta(4)=\beta(6)=\beta(7)=\beta(9)=\beta(10)=\beta(11)=2 \Rightarrow w(4)=30;\, w(6)=32;\, w(7)=33;\, w(9)=35;\, w(10)=36;\, w(11)=37$;
			\item $12=f_6+f_4+f_2 \Rightarrow \beta(12)=3 \Rightarrow w(12)=51$.
		\end{itemize}
	\end{example}

	\section{The Frobenius number of $S(a)$}\label{frobeniusS(a)}
	
	The main aim in this section is to prove Theorem~\ref{thm17}, which provides us a formula for the Frobenius number of $S(a)$ as a function of $a$ and $f_a$. For this we need some previous results.
	
	If $x\in\mathbb{N}$, then we denote by $\gamma(x) = \max\{ l\in\mathbb{N} \mid f_l\leq x \}$.
	
	\begin{remark}\label{rem01}
		By Lemma~\ref{lem00}, it is clear that, if $x=\sum_{i=2}^k b_i f_i$ is the Zeckendorf decomposition of $x\in\mathbb{N}\setminus\{0\}$, then $\gamma(x)=k$. Moreover, $\gamma(0)=0$.
	\end{remark}
	
	The following result is an immediate consequence of Remarks~\ref{rem00} and \ref{rem01} and the definitions of $\beta(x)$ and $\gamma(x)$.
	
	\begin{lemma}\label{lem13}
		If $x\in\mathbb{N}\setminus\{0\}$, then $\beta(x) = \beta\left(x-f_{\gamma(x)}\right) + 1$.
	\end{lemma}
	
	Since Zeckendorf decompositions do not admit consecutive Fibonacci numbers as addends, we easily have the following result.
	
	\begin{lemma}\label{lem13b}
		If $x\in\mathbb{N}\setminus\{0\}$, then $\gamma\left(x-f_{\gamma(x)}\right) \leq \gamma(x) - 2$.
	\end{lemma}
	
	We can give $\beta(x)$ very easily in some cases. For example, if  $a\in\mathbb{N}\setminus\{0\}$, then $\beta(f_a)=1$. Let us see another case. As usual, $\lfloor x \rfloor = \max\{z\in\mathbb{Z} \mid z\leq x \}$.
	 
	\begin{lemma}\label{lem14}
		If $a\in\mathbb{N}\setminus\{0\}$, then $\beta(f_a-1) = \left\lfloor \frac{a-1}{2} \right\rfloor$.
	\end{lemma}
	
	\begin{proof}
		We will argue by mathematical induction on $a$. First, it is trivial that the result is true for $a\in\{1,2\}$.
		
		Now, by Lemma~\ref{lem13}, if $a\geq 3$, then $\beta(f_a - 1) = \beta(f_a - 1 - f_{a-1})+ 1 = \beta(f_{a-2} - 1) + 1$. Therefore, by the induction hypothesis on $a-2$, we have that $\beta(f_a - 1) = \left\lfloor \frac{a-3}{2} \right\rfloor + 1 = \left\lfloor \frac{a-1}{2} \right\rfloor$.
	\end{proof}
	
	In the general case, we can show an upper bound.
	
	\begin{lemma}\label{lem15}
		If $x\in\mathbb{N}$, then $\beta(x) \leq \left\lfloor \frac{\gamma(x)}{2} \right\rfloor$.
	\end{lemma}
	
	\begin{proof}
		We will use mathematical induction on $x$. First, the result is trivially true for $x\in\{0,1,2\}$.
		
		Now, let us suppose that $x\geq 3$ and that $\beta(y) \leq \left\lfloor \frac{\gamma(y)}{2} \right\rfloor$ for all $y<x$. Then, by Lemmas~\ref{lem13} and \ref{lem13b}, we have that
		\[ \beta(x) = \beta\left(x-f_{\gamma(x)}\right) + 1 \leq  \left\lfloor \frac{\gamma\left(x-f_{\gamma(x)}\right)}{2} \right\rfloor + 1 \leq \left\lfloor \frac{\gamma(x)-2}{2} \right\rfloor + 1 = \left\lfloor \frac{\gamma(x)}{2} \right\rfloor. \]
	\end{proof}
	
%

	We are ready to show the announced theorem.
	
	\begin{theorem}\label{thm17}
		If $a\in\mathbb{N}\setminus\{0,1,2\}$, then $\mathrm{F}(S(a)) =\left\lfloor \frac{a-1}{2} \right\rfloor f_a -1$.
	\end{theorem}
	
	\begin{proof}
		If $x\in\{0,1,\ldots,f_a-1\}$, then $\gamma(x)\leq a-1$. Thus, from Theorem~\ref{thm12} and Lemmas~\ref{lem14} and \ref{lem15}, we deduce that $\max\left( \mathrm{Ap}(S(a),f_a)\right) = \left\lfloor \frac{a-1}{2} \right\rfloor f_a + f_a -1$. Now, by Proposition~\ref{prop03}, we conclude that $\mathrm{F}(S(a)) = \left\lfloor \frac{a-1}{2} \right\rfloor f_a - 1$.
	\end{proof}
	
	\begin{example}\label{exmp18}
		By Example~\ref{exmp07}, we have that $S(7)=\langle 13,14,15,16,18,21 \rangle$. From Theorem~\ref{thm17}, we get that $\mathrm{F}(S(7)) = \left\lfloor \frac{7-1}{2} \right\rfloor f_7 - 1 = 38$.
	\end{example}
	
	Since $\mathrm{e}(S(a))=a-1$ and $\mathrm{m}(S(a))=f_a$, we can reformulate Theorem~\ref{thm17} as follows.
	
	\begin{corollary}\label{cor19}
		If $a\in\mathbb{N}\setminus\{0,1,2\}$, then $\mathrm{F}(S(a)) =\left\lfloor \frac{\mathrm{e}(S(a))}{2} \right\rfloor \mathrm{m}(S(a)) -1$.
	\end{corollary}
	
	\begin{remark}
		It is easy to check that Theorem~\ref{thm17} and Corollary~\ref{cor19} are also true for $a=2$.
	\end{remark}

	\section{The genus of $S(a)$}\label{genusS(a)}
	
	In this section we will give a formula for the genus of $S(a)$. As usual, if $A$ is a set, then we denote by $\#(A)$ the cardinality of $A$. Moreover, if $m,n\in\mathbb{N}$ with $m\leq n-2$, then we denote by $\mathcal{F}_n(m)$ the set
	\[ \{ X\subseteq \{2,\ldots,n-1\} \mid  \#(X)=m \mbox{ and no two consecutive integers belong to } X \}. \]
	It is clear that $\#(\mathcal{F}_n(m))=0$ for all $m > \frac{n-1}{2}$. In other case, we have a classical result on counting subsets. 
	\begin{lemma}{\cite[Lemma~1]{kaplansky}}\label{lem20}
		If $m,n\in\mathbb{N}\setminus\{0\}$ and $m\leq \frac{n-1}{2}$, then $\#(\mathcal{F}_n(m)) = \binom{n-1-m}{m}$.
	\end{lemma}
	
	\begin{remark}\label{rem21}
	The Zeckendorf decomposition gives us a bijection between the sets 
	$\{1,\ldots,f_a-1\}$ and $\mathcal{F}(a) = \mathcal{F}_a(1) \cup \cdots \cup \mathcal{F}_a\left( \left\lfloor \frac{a-1}{2} \right\rfloor \right)$. Indeed, if $x\in\{1,\ldots,f_a-1\}$ has the Zeckendorf decomposition $\sum_{i=2}^k b_i f_i$ ($k<a$, $(b_2,\ldots,b_k)\in\{0,1\}^{k-1}$, $b_k=1$, and $b_ib_{i+1}=0$ for all $i\in\{2,\ldots,k-1\}$), then we can associate $x$ with the set $B(x) \in \mathcal{F}_a(\beta(x))$ consisting of all subscripts $j$ such that $b_j=1$. Now, from the well known equality $f_a=\sum_{j=0}^{\left\lfloor \frac{a-1}{2} \right\rfloor} \binom{a-1-j}{j}$ and the uniqueness of the Zeckendorf decomposition, the correspondence associating $x$ to $B(x)$ is the sought bijection.	
	\end{remark}
	
	As a consequence of Theorem~\ref{thm12}, Lemma~\ref{lem20}, and Remark~\ref{rem21}, we have the following result.
	
	\begin{proposition}\label{prop22}
		If $a\in\mathbb{N}\setminus\{0,1,2\}$, then
		\[ \mathrm{Ap}\left(S(a),f_a\right)\setminus\{0\} = \left\{ \left(\#(B)\right)f_a + \sum_{b\in B} f_b \mid B \in \mathcal{F}(a)\setminus\{\emptyset\} \right\}. \]
		Moreover, if $\{B_1,B_2\} \subseteq \mathcal{F}(a)\setminus\{\emptyset\}$, then $\left(\#(B_1)\right)f_a + \sum_{b\in B_1} f_b = \left(\#(B_2)\right)f_a + \sum_{b\in B_2} f_b$ if and only if $B_1=B_2$.
	\end{proposition}
	
	The next result is an easy consequence of Proposition~\ref{prop03}.
	
	\begin{lemma}\label{lem23}
		If $S$ is a numerical semigroup, $n\in S\setminus\{0\}$, $\{k_1,k_2,\ldots,k_{n-1}\}\subseteq\mathbb{N}$, and $\mathrm{Ap}(S,n)=\{0,k_1n+1,k_2n+2,\ldots,k_{n-1}n+n-1\}$, then $\mathrm{g}(S)=k_1+k_2+\cdots+k_{n-1}$.
	\end{lemma}
	
	By Theorem~\ref{thm12} and Lemma~\ref{lem23}, we can deduce the following result.
	
	\begin{lemma}\label{lem24}
		If $a\in\mathbb{N}\setminus\{0,1,2\}$, then $\mathrm{g}(S(a)) = \sum_{x=1}^{f_a-1} \beta(x)$.
	\end{lemma}
	
	Let $B(x)$ be the set associated to $x\in\{1,\ldots,f_a-1\}$ in Remark~\ref{rem21}. Then it is clear that $\#(B(x)) = \beta(x)$. This fact, together Proposition~\ref{prop22} and Lemma~\ref{lem24}, leads to the next result.
	
	\begin{proposition}\label{prop27}
		If $a\in\mathbb{N}\setminus\{0,1,2\}$, then $\mathrm{g}(S(a)) = \sum_{i=1}^{\left\lfloor \frac{a-1}{2} \right\rfloor} i\binom{a-1-i}{i}$.
	\end{proposition}

	In fact, we can explicitly compute the summation of the above proposition.
	
	\begin{theorem}\label{thm28}
		If $a\in\mathbb{N}\setminus\{0,1,2\}$, then $\mathrm{g}(S(a)) = \frac{a-2}{5} f_a + \frac{a}{5} f_{a-2}$.
	\end{theorem}
	
	\begin{proof}
		Let us first see that, if $a\geq5$, then
		\[ \mathrm{g}(S(a)) = \mathrm{g}(S(a-1)) + \mathrm{g}(S(a-2)) + f_{a-2}. \]
		
		Let us take $a=2k+3$ for $k\in\mathbb{N}\setminus\{0\}$. Then, by Proposition~\ref{prop27}, we have that $\mathrm{g}(S(a)) = \mathrm{g}(S(2k+3)) = \sum_{i=1}^{k+1} i\binom{2k+2-i}{i}$ and, hereafter,
		\[ \begin{split}
			\mathrm{g}(S(2k+3)) & = \sum_{i=1}^{k} i \left[ \binom{2k+1-i}{i} + \binom{2k+1-i}{i-1} \right] + (k+1)\binom{k+1}{k+1} \\
			& = \sum_{i=1}^{k} i \binom{2k+1-i}{i} + \sum_{i=1}^{k} i \binom{2k+1-i}{i-1} + (k+1)\binom{k}{k} \\
			& = \mathrm{g}(S(2k+2)) + \sum_{i=1}^{k+1} i \binom{2k+1-i}{i-1} \\
			& = \mathrm{g}(S(2k+2)) + \sum_{i=0}^{k} i \binom{2k-i}{i} + \sum_{i=0}^{k} \binom{2k-i}{i} \\
			& = \mathrm{g}(S(2k+2)) + \mathrm{g}(S(2k+1)) + f_{2k+1}.
		\end{split} \]
		
		If $a=2k+4$, with $k\in\mathbb{N}\setminus\{0\}$, the equality check is similar, so we omit it.
		
		To conclude that $\mathrm{g}(S(a)) = \frac{a-2}{5} f_a + \frac{a}{5} f_{a-2}$, we use mathematical induction. In fact, we easily have the equality for $a=3$ and $a=4$. Now, if $a\geq5$ and we assume that the equality is true for all $k\in\{3,4,\ldots,a-1\}$, then
		\[ \begin{split}
			\mathrm{g}(S(a)) & = \mathrm{g}(S(a-1)) + \mathrm{g}(S(a-2)) + f_{a-2} \\
			& = \frac{a-3}{5} f_{a-1} + \frac{a-1}{5} f_{a-3} + \frac{a-4}{5} f_{a-2} + \frac{a-2}{5} f_{a-4} + f_{a-2} \\
			& = \frac{a-2}{5} \left( f_{a-1} + f_{a-2} \right) + \frac{a}{5} \left( f_{a-3} + f_{a-4} \right) - \frac{f_{a-1} - 3f_{a-2} + f_{a-3} + 2f_{a-4}}{5} \\
			& = \frac{a-2}{5} f_a + \frac{a}{5} f_{a-2},
		\end{split} \]
		considering that $f_{a-1} - 3f_{a-2} + f_{a-3} + 2f_{a-4} = -2f_{a-2} + 2f_{a-3} + 2f_{a-4} = 0$.
	\end{proof}
	
	\begin{example}\label{exmp29}
		By Example~\ref{exmp07}, we know that $S(7)=\langle 13,14,15,16,18,21 \rangle$. From Theorem~\ref{thm28}, we have that $\mathrm{g}(S(7)) = f_7+\frac{7}{5}f_5 = 13+7 =20$.
	\end{example}
	
	\begin{remark}
		It is easy to check that Theorem~\ref{thm28} is also true for $a=2$.
	\end{remark}
	
	Since we know explicitly expressions for the embedding dimension, the Frobenius number, and the genus of $S(a)$, we can check that this family of numerical semigroups satisfies the Wilf's conjecture (see \cite{wilf}). If $S$ is a numerical semigroup, then we denote by $\mathrm{n}(S)$ the cardinality of the set $\{s\in S \mid s<\mathrm{F}(S) \}$.
	
	\begin{corollary}\label{cor30}
		If $a\in\mathbb{N}$, then $\mathrm{F}(S(a)) + 1 \leq \mathrm{e}(S(a)) \mathrm{n}(S(a))$. 
	\end{corollary}
	
	\begin{proof}
		If $a\in\{0,1,2\}$, then $S(a) = \mathbb{N}$ and, therefore, the result is obvious.
		
		If $a\geq3$, we use an equivalent inequality. Indeed, since $\mathrm{g}(S) + \mathrm{n}(S) = \mathrm{F}(S) + 1$ for any numerical semigroup $S$, then
		\[ \mathrm{F}(S) + 1 \leq \mathrm{e}(S) \mathrm{n}(S) \Leftrightarrow \mathrm{e}(S) \mathrm{g}(S) \leq \left( \mathrm{e}(S)-1 \right) \left( \mathrm{F}(S)+1 \right). \]
		Now, by Corollary~\ref{cor06}, Theorem~\ref{thm17}, and Theorem~\ref{thm28}, we have that 
		\[ \mathrm{e}(S(a)) \mathrm{g}(S(a)) \leq \left( \mathrm{e}(S(a))-1 \right) \left( \mathrm{F}(S(a))+1 \right) \Leftrightarrow \]
		\[ (a-1) \frac{(a-2)f_a+af_{a-2}}{5} \leq (a-2) \left\lfloor \frac{a-1}{2} \right\rfloor f_a. \]
		By direct verification, the last inequality is true for $a\in\{3,\ldots,10\}$. If $a\geq11$, since $f_{a-2}\leq f_a$, it is enough to see that $2(a-1)^2\leq \frac{5}{2}(a-2)^2$, which is equivalent to $20\leq (a-6)^2$.
	\end{proof}

	\section*{Acknowledgement}
	
	Both authors are supported by the project MTM2017-84890-P (funded by Mi\-nis\-terio de Econom\'{\i}a, Industria y Competitividad and Fondo Europeo de Desarrollo Regional FEDER) and by the Junta de Andaluc\'{\i}a Grant Number FQM-343.

\end{document}